\documentclass[11pt]{article}
\usepackage{amsmath}
\usepackage{amssymb}
\usepackage{latexsym,bm}
\usepackage{amsthm}
\usepackage{graphicx}

\usepackage[top=2.4cm,bottom=2cm,left=2.4cm,right=2.4cm]{geometry}
\usepackage{algorithm}
\usepackage{algorithmic}
\usepackage{cite}
\allowdisplaybreaks
\usepackage{multirow}
\usepackage{threeparttable}
\usepackage{slashbox}
\usepackage{amsmath}
\usepackage{subfigure}

\newtheorem{theorem}{Theorem}[section]

\newtheorem{lemma}{Lemma}[section]

\theoremstyle{definition}
\newtheorem{definition}{Definition}[section]
\newtheorem{remark}{Remark}[section]

\theoremstyle{example}

\numberwithin{equation}{section}

\begin{document}

\title{An outer reflected forward-backward splitting algorithm for solving monotone inclusions}

\author{ Hui Yu$^{a}$, Chunxiang Zong$^{b}$, Yuchao Tang$^{c}$\footnote{Corresponding author. Email: hhaaoo1331@163.com (Yuchao Tang)}
\\
\\
%EndAName
{\small ${^a}$ School of Information Engineering, Nanchang University,  Nanchang 330031, P.R. China }\\
{\small $^{b}$ School of Mathematics and Statistics, Lanzhou University, Lanzhou 730000, P.R. China }\\
{\small $^{c}$ Department of Mathematics, Nanchang University,  Nanchang 330031, P.R. China}}

 \date{}

\maketitle

{}\textbf{Abstract.}\ Monotone inclusions have wide applications in solving various convex optimization problems arising in signal and image processing, machine learning, and medical image reconstruction. In this paper, we propose a new splitting algorithm for finding a zero of the sum of a maximally monotone operator, a monotone Lipschitzian operator, and a cocoercive operator, which is called outer reflected forward-backward splitting algorithm. Under mild conditions on the iterative parameters, we prove the convergence of the proposed algorithm. As applications, we employ the proposed algorithm to solve composite monotone inclusions involving monotone Lipschitzian operator, cocoercive operator, and the parallel sum of operators. The advantage of the obtained algorithm is that it is a completely splitting algorithm, in which the Lipschitzian operator and the cocoercive operator are processed via explicit steps and the maximally monotone operators are processed via their resolvents.



%%%%%%%%%%%%%%%%%%%%%%%%%%%%%%%%%%%%%%%%%%%%%%%%%%%%%%%%%%%%%%%%%%%%%%%%%%%%%%%%%%%%%%%%%%%%%%%%%%%%%%%%%%%%%%%%%%%%%%%%%%%%%
\section{Introduction}
\label{sec-introduction}

In the last two decades, due to the large-scale and rich structure of problems in signal and image processing, machine learning, and medical image reconstruction, first-order optimization methods that only use function values and gradient information have been paid much attention by researchers. Monotone inclusions paly an important role in studying first-order optimization algorithms. In this paper, we are interested in solving the following three-operator monotone inclusions problem.

 \textbf{Problem 1.} Let $H$ be a real Hilbert space. Let $A:H\rightarrow 2^H$ be a maximally monotone operator. Let $B:H\rightarrow H$ be a monotone $L$-Lipschitz operator and let $C:H\rightarrow H$ be a $\beta$-cocoercive operator, for some $L>0$ and $\beta >0$. The problem is to
\begin{equation}\label{three-sum-monotone}
\textrm{ find } x\in H, \textrm{ such that } 0\in Ax + Bx + Cx.
\end{equation}

This monotone inclusions (\ref{three-sum-monotone}) includes some special cases. For example, when $B=0$, (\ref{three-sum-monotone}) becomes
\begin{equation}\label{two-sum-monotone1}
\textrm{ find } x\in H, \textrm{ such that } 0\in Ax + Cx.
\end{equation}
If $C=0$, (\ref{three-sum-monotone}) reduces to
\begin{equation}\label{two-sum-monotone2}
\textrm{ find } x\in H, \textrm{ such that } 0\in Ax + Bx.
\end{equation}
If $B=0$ and $C=0$, (\ref{three-sum-monotone}) reduces to the simple monotone inclusion
\begin{equation}\label{one-monotone}
\textrm{ find } x\in H, \textrm{ such that } 0\in Ax.
\end{equation}

The forward-backward splitting (FBS) algorithm is a popular algorithm for solving (\ref{two-sum-monotone1}), which was introduced by Lions and Mercier \cite{Lions1979SIAM} and Passty \cite{passty1979JMAA}. The classical FBS algorithm reads as
\begin{equation}\label{forward-backward-splitting}
x_{k+1} = J_{\lambda A}(x_k - \lambda Cx_k),
\end{equation}
where $\lambda \in (0,2\beta)$. The convergence of the FBS algorithm (\ref{forward-backward-splitting}) and its variants have been extensively studied in \cite{Chenhg1997,combettes2005,Combettes2014Optimization,Combettes2015JMAA,lorenz2015JMIV,CuiJIA2019}. In particular, Raguet et al. \cite{Raguet-SIAM-2013} proposed a generalized forward-backward splitting algorithm for finding a zero of the sum of a cocercive operator and a finite family of maximally monotone operators. Some related works on the generalized forward-backward splitting algorithm  can be found in \cite{Raduet2015SIAMJIS,briceno2015Optim,davis2015,Zong2018}.

The forward-backward-forward splitting (FBFS) algorithm proposed by Tseng \cite{Tseng2000SIAM} is the first algorithm for solving (\ref{two-sum-monotone2}). The FBFS algorithm is defined as
\begin{equation}\label{forward-backward-forward-splitting}
\left \{
\begin{aligned}
u_k & = J_{\lambda A}(x_k - \lambda Bx_k) \\
x_{k+1} & = u_k + \lambda Bx_k - \lambda Bu_k,
\end{aligned}\right.
\end{equation}
where $\lambda \in (0,\frac{1}{L})$. The FBFS algorithm (\ref{forward-backward-forward-splitting}) was generalized by several authors, such as relaxed FBFS algorithm \cite{Dong2003,Dong2014AMC}, inexact FBFS algorithm \cite{briceno2011SIAM}, inertial FBFS algorithm \cite{Bot2016NA,Gibali2020CAM}, variable metric FBFS algorithm \cite{Vu2013NFAO} and stochastic FBFS algorithm \cite{Vu2016OL}.

For the considered three-operator monotone inclusions (\ref{three-sum-monotone}), Brice\~{n}o-Arias and Davis \cite{Arias2017A} first proposed a so-called forward-backward-half forward splitting (FBHFS) algorithm, which is defined by
\begin{equation}\label{forward-backward-half-forward-splitting}
\left \{
\begin{aligned}
u_k & = J_{\lambda A}(x_k - \lambda (B+C)x_k) \\
x_{k+1} & = u_k + \lambda Bx_k - \lambda Bu_k,
\end{aligned}\right.
\end{equation}
where $\lambda \in (0, \frac{4\beta}{1+\sqrt{1+16\beta^2 L^2}})$. They proved the convergence of (\ref{forward-backward-half-forward-splitting}) with variable step-size and line search procedure. It is easy to see that this iteration scheme (\ref{forward-backward-half-forward-splitting}) becomes the forward-backwards splitting algorithm (\ref{forward-backward-splitting}) when $B=0$, and it reduces to the forward-backward-forward splitting algorithm (\ref{forward-backward-forward-splitting}) when $C=0$. Recently, Malitsky and Tam \cite{Malitsky2020SIAMJO} proposed the following semi-forward-reflected-backward splitting (SFRBS) algorithm for solving the three-operator monotone inclusions (\ref{three-sum-monotone}),
\begin{equation}\label{Malitsky-Tam-splitting}
x_{k+1} = J_{\lambda A}(x_k - 2\lambda Bx_k + \lambda Bx_{k-1} - \lambda Cx_k).
\end{equation}
They proved the convergence of (\ref{Malitsky-Tam-splitting}) under the assumption that $\lambda \in (0,\frac{2\beta}{4\beta L +1})$. When $C=0$, (\ref{Malitsky-Tam-splitting}) reduces to the forward-reflected-backward splitting (FRBS) algorithm,
\begin{equation}\label{forward-reflected-backward-splitting}
x_{k+1} = J_{\lambda A}(x_k - 2\lambda Bx_k + \lambda Bx_{k-1}),
\end{equation}
which is also introduced in \cite{Malitsky2020SIAMJO}. When $B=0$, (\ref{Malitsky-Tam-splitting}) becomes the forward-backward splitting algorithm (\ref{forward-backward-splitting}). In contrast, Cevher and V\~{u} \cite{Cevher-2019-arxiv} proposed a semi-reflected forward-backward splitting (SRFBS) algorithm for solving (\ref{three-sum-monotone}), which is defined by
\begin{equation}\label{semi-reflected-forward-backward-splitting}
x_{k+1} = J_{\lambda A }(x_k - \lambda B (2x_k - x_{k-1}) - \lambda Cx_k ).
\end{equation}
The convergence of (\ref{semi-reflected-forward-backward-splitting}) is studied in \cite{Cevher-2019-arxiv} under suitable conditions on the iterative parameter $\lambda$. When $B=0$, (\ref{semi-reflected-forward-backward-splitting}) reduces to the forward-backward splitting algorithm (\ref{forward-backward-splitting}). When $C=0$, (\ref{semi-reflected-forward-backward-splitting}) becomes the reflected forward-backward splitting (RFBS) algorithm,
\begin{equation}\label{reflected-forward-backward-splitting}
x_{k+1} = J_{\lambda A }(x_k - \lambda B (2x_k - x_{k-1})),
\end{equation}
which is introduced in \cite{Cevher2020SVVA}. In particular, if $B$ is linear, (\ref{reflected-forward-backward-splitting}) is equivalent to (\ref{forward-reflected-backward-splitting}). Some recent generalization of the forward-reflected-backward splitting algorithm (\ref{forward-reflected-backward-splitting}) and  the reflected forward-backward splitting  algorithm (\ref{reflected-forward-backward-splitting}) can be found in \cite{Gibali2020Symmetry,Hieu2020BIMS,Hieu20204OR}.

Different from the FBFS algorithm (\ref{forward-backward-forward-splitting}), the FRBS algorithm (\ref{forward-reflected-backward-splitting}), and the RFBS algorithm (\ref{reflected-forward-backward-splitting}), Csetnek et al. \cite{Csetnek2019AMO} proposed a new splitting algorithm for solving (\ref{two-sum-monotone2}), which is derived from a non-standard discretization of a continuous dynamics associated with the Douglas-Rachford splitting algorithm. The detailed of this algorithm is presented below,
\begin{equation}\label{Csetnek-splitting}
x_{k+1} = J_{\lambda A}(x_k - \lambda Bx_k) - \lambda (Bx_k - Bx_{k-1}),
\end{equation}
which was proved converge weakly to a solution of (\ref{two-sum-monotone2}) when $\lambda \in (0,\frac{1}{3L})$. Since a cocoercive operator is also a monotone and Lipschitz continuous operator, then it is naturally to apply (\ref{Csetnek-splitting}) to solve the considered three-operator monotone inclusions (\ref{three-sum-monotone}). The resulting algorithm is
\begin{equation}\label{Csetnek-splitting-three}
x_{k+1} = J_{\lambda A}(x_k - \lambda Bx_k -\lambda Cx_k) - \lambda ((B+C)x_k - (B+C)x_{k-1}).
\end{equation}
However, this iteration scheme (\ref{Csetnek-splitting-three}) doesn't make full use of the cocoercivity property of the operator $C$. To overcome this shortage, the purpose of this paper is to introduce a new splitting algorithm to solve (\ref{three-sum-monotone}). It not only includes the classical forward-backward splitting algorithm (\ref{forward-backward-splitting}), but also has the iterative algorithm (\ref{Csetnek-splitting}) as its special case. Follow the idea of \cite{Csetnek2019AMO}, we study the convergence of the proposed algorithm. Furthermore, we apply the proposed algorithm to solve a primal-dual pair of  composite monotone inclusions problem.

The paper is organized as follows. In Section 2, we provide the notation and some elements of the theory of maximally monotone operators. In Section 3, we introduce the new splitting algorithm and prove its convergence. In Section 4, we address the primal-dual pair of  composite monotone inclusions problem. Finally, we give some conclusions and future works.

%%%%%%%%%%%%%%%%%%%%%%%%%%%%%%%%%%%%%%%%%%%%%%%%%%%%%%%%%%%%%%%%%%%%%%%%%%%%%%%%%%%%%%%%%%%%%%%%%%%%%%%%%5
\section{Preliminaries}
\label{preli}

In this section, we introduce some basic notations and preliminary results in monotone operator theory, which will be used throughout the paper. Most of them can be found in \cite{bauschkebook2017}.

Let $H$ be a real Hilbert space with scalar product denoted by $\langle\cdot,\cdot\rangle$ and  associated norm by $\|\cdot\|$, respectively. Weak and strong convergence is defined as symbols $\rightharpoonup$ and $\rightarrow$, respectively. $I$ denotes the identity operator on $H$.

Let $A:H\rightarrow 2^{H}$ be a set-valued operator. The domain of $A$ is defined by $dom A = \{x\in H\mid Ax \neq \emptyset\}$ and the graph of $A$ is defined by $gra A=\{(x,u)\in H\times H\mid u\in Ax\}$. The set of zeros of $A$ is $zer A=\{x\in H\mid 0\in Ax\}$ and the range of $A$ is $ran A=\{u\in H\mid(\exists x\in H)u\in Ax\}$.

\begin{definition}(\cite{bauschkebook2017})(Monotone operators and maximally monotone operators)\label{maxim-def}
Let $A:H\rightarrow 2^{H}$ be a set-valued operator. $A$ is said to be monotone, if for any $(x,u),(y,v)\in gra A$ it holds that
$$
\langle x-y,u-v\rangle\geq 0.
$$
Further, $A$ is said to be maximally monotone, if the graph of $A$ is not a proper subset of the graph of any other monotone operator.
\end{definition}

\begin{definition}(\cite{bauschkebook2017})(Resolvent operators)
Let $A:H\rightarrow 2^{H}$ be a set-valued operator. The resolvent of a monotone operator $A$ with index $\lambda >0$ is defined as
$$
  J_{ \lambda A}=(I+ \lambda A)^{-1}.
$$
\end{definition}

\begin{definition}(\cite{bauschkebook2017})\label{nonexp-def}
Let $T:H\rightarrow H$ be a single-valued operator.

(i)\, $T$ is said to be $L$-Lipschitz continuous, for some $L>0$, if it holds that
$$
\|Tx - Ty\| \leq L\|x-y\|, \forall x,y \in H.
$$
If $L=1$, then $T$ is called nonexpansive.

(ii)\, $T$ is said to be $\beta$-cocoercive (or $\beta$-inverse strongly monotone),  if it holds that
$$
\langle x -y, Tx-Ty \rangle \geq \beta \|Tx-Ty\|^2,  \forall x,y \in H,
$$
where $\beta >0$.
\end{definition}

\begin{remark}
Notice that a cocoercive operator is  Lipschitz continuous and monotone operator.
\end{remark}

 \begin{definition}(\cite{bauschkebook2017})(Parallel sum)
For any $i=1,\cdots, m$, let $C_{i}:H\rightarrow 2^{H}$ be set-valued operator.
 The parallel sum $C_{1}\Box \cdots \Box C_{m}$ is defined as
$$
C_{1}\Box \cdots \Box C_{m} = (C_{1}^{-1}+\cdots +C_{m}^{-1})^{-1}.
$$
\end{definition}

The following lemmas will be used to prove the main convergence.

\begin{lemma}(\cite{bauschkebook2017})\label{lemma1}
Let $\{x_k\}$ be a sequence in $H$. Then $\{x_k\}$ converges weakly if and only if it is bounded and possesses at most one weak sequential cluster point.
\end{lemma}

\begin{lemma}(\cite{bauschkebook2017})(Opial's lemma)\label{lemma2}
Let $C$ be a nonempty set of $H$ and $\{x_k\}$ be a sequence in $H$ such that the following two conditions hold:

\emph{(i)} for any $x\in C$, $\lim_{k\rightarrow +\infty}\|x_k - x\|$ exists;

\emph{(ii)} every sequential weak cluster point of $\{x_k\}$ is in $C$.

Then $\{x_k\}$ converges weakly to a point in $C$.
\end{lemma}

%%%%%%%%%%%%%%%%%%%%%%%%%%%%%%%%%%%%%%%%%%%%%%%%%%%%%%%%%%%%%%%%%%%%%%%%%%%%%%%%%%%%%%%%%%%%%%%%%%%%%%%%%%%%%%%%%%%%%%%%%%%%%%%%%%%%%%%%
\section{Main algorithm and convergence analysis}

In this section, we propose our main algorithm and we prove its convergence to a solution of the three-operator monotone inclusions (\ref{three-sum-monotone}). First, we prove the following lemma.

\begin{lemma}\label{main-lemma}
Let $A:H\rightarrow 2^{H}$ be a maximally monotone operator. Let $\{y_k\}\subseteq H$ and $\{z_k\}\subseteq H$. Let $x_0\in H$, and set
\begin{equation}\label{eq3-1}
x_{k+1} = J_{\lambda A}(x_k - y_k - z_k) - (y_k - y_{k-1}),
\end{equation}
where $\lambda >0$. Then for all $x\in H$ and $y_1 + y_2 \in - \lambda Ax$, we have
\begin{equation}\label{eq3-2}
\begin{aligned}
\| (x_{k+1}+y_k) - (x + y_1) \|^2 & \leq \| (x_k + y_{k-1}) - (x+y_1) \|^2 + 4 \langle y_k - y_{k-1}, x_k - x_{k-1} \rangle - \|x_{k+1} - x_k\|^2 \\
 & \quad - 3 \|y_k - y_{k-1}\|^2 -2 \langle y_k - y_1, x_k - x \rangle + 2 \langle z_k - y_2, x - x_{k+1} \rangle \\
 & \quad + 2 \langle y_2 - z_k, y_k - y_{k-1} \rangle.
\end{aligned}
\end{equation}

\end{lemma}

\begin{proof}
It follows from the definition of the resolvent and (\ref{eq3-1}), we have
\begin{equation}\label{eq3-3}
x_k - y_k - z_k - x_{k+1} - (y_k - y_{k-1}) \in \lambda A (x_{k+1} + y_k -y_{k-1}).
\end{equation}
Since $-y_1 -y_2 \in \lambda Ax$ and $\lambda A$ is monotone, we have
\begin{equation}\label{eq3-4}
\begin{aligned}
0 & \leq \langle x_{k+1} - x_k + y_k - y_1 + z_k -y_2 + y_k - y_{k-1}, x -x_{k+1} - (y_k - y_{k-1}) \rangle \\
& = \langle x_{k+1} - x_k, x - x_{k+1} \rangle + \langle y_k - y_1, x - x_{k+1} \rangle + \langle z_k - y_2, x - x_{k+1} \rangle \\
& \quad + \langle y_k - y_{k-1}, x - x_{k+1} \rangle + \langle x_k - x_{k+1}, y_k - y_{k-1} \rangle + \langle y_1 - y_k, y_k - y_{k-1} \rangle \\
& \quad + \langle y_2 - z_k, y_k - y_{k-1} \rangle - \|y_k - y_{k-1}\|^2.
\end{aligned}
\end{equation}

It is observed that
\begin{equation}\label{eq3-5}
2 \langle x_{k+1} - x_k, x - x_{k+1} \rangle = \|x_k - x\|^2 - \| x_{k+1} - x_k \|^2 - \| x_{k+1} -x \|^2,
\end{equation}
\begin{equation}\label{eq3-6}
2 \langle y_1 - y_k, y_k - y_{k-1} \rangle = \| y_{k-1} - y_1 \|^2 - \| y_k - y_{k-1} \|^2 - \| y_k - y_1 \|^2,
\end{equation}
and
\begin{equation}\label{eq3-7}
\begin{aligned}
\langle y_k - y_{k-1},x - x_{k+1} \rangle & = \langle y_k - y_{k-1}, x_k - x_{k+1} \rangle + \langle y_{k-1} - y_1, x_k -x \rangle \\
& \quad + \langle y_1 - y_k, x_k -x \rangle.
\end{aligned}
\end{equation}
Substituting (\ref{eq3-5}), (\ref{eq3-6}) and (\ref{eq3-7}) into (\ref{eq3-4}), we obtain
\begin{align}\label{eq3-8}
& \quad \| x_{k+1} -x \|^2 + 2 \langle y_k - y_1, x_{k+1} - x \rangle + \|y_k - y_1\|^2 \nonumber  \\
& \leq \|x_k -x\|^2 + 2 \langle y_{k-1} - y_1, x_k -x \rangle + \| y_{k-1} - y_1 \|^2 \nonumber \\
& \quad + 4 \langle y_k - y_{k-1}, x_k - x_{k+1} \rangle - \| x_{k+1} - x_k \|^2 - 3 \| y_k - y_{k-1} \|^2 \nonumber \\
& \quad - 2 \langle y_k - y_1, x_k -x \rangle + 2 \langle z_k - y_2, x - x_{k+1} \rangle + 2 \langle y_2 - z_k, y_k - y_{k-1} \rangle,
\end{align}
which is equivalent to (\ref{eq3-2}). This completes the proof.

\end{proof}

We are ready to present the main convergence theorem.

\begin{theorem}\label{main-theorem}
Let $A:H\rightarrow 2^H$ be maximally monotone operator, $B:H\rightarrow H$ be monotone and $L$-Lipschitz, for some $L>0$, and $C:H\rightarrow H$ be $\beta$-cocoercive, for some $\beta >0$. Suppose that $zer(A+B+C)\neq \emptyset$. Let $x_0, x_{-1}\in H$, and set
\begin{equation}\label{main-algorithm}
x_{k+1} = J_{\lambda A}(x_k - \lambda Bx_k - \lambda Cx_k) - \lambda (Bx_k - Bx_{k-1}),
\end{equation}
where $\lambda$ satisfies the following conditions:
\begin{equation}
\begin{aligned}
& 0 < \lambda < (2\beta - \varepsilon_2)\varepsilon_1 \\
& 0 < \lambda \leq (3 - \varepsilon_3)\varepsilon_2 \\
& 0 < \lambda < \frac{1/2 - \varepsilon_1 - 1/\varepsilon_3}{L}, \\
\end{aligned}
\end{equation}
for some $\varepsilon_1 >0$, $\varepsilon_2 >0$ and  $\varepsilon_3 >0$ such that $\varepsilon_2 < 2\beta $, $2 < \varepsilon_3 <3$ and $\epsilon_1 + 1/\varepsilon_3 < \frac{1}{2}$. Then, there exists a point $x\in zer(A+B+C)$ such that

\emph{(i)} $\sum_{k=0}^{+\infty}\|Cx_k - Cx\|^2 < +\infty$;

\emph{(ii)} $\{x_k\}$ converges weakly to $x$;

\emph{(iii)} $\{Bx_k\}$ converges weakly to $Bx$.

\end{theorem}

\begin{proof}
(i) Let $x\in zer(A+B+C)$. Let $y_1 = \lambda Bx$ and $y_2 = \lambda Cx$ such that $y_1 + y_2\in -\lambda Ax$. In Lemma \ref{main-lemma}, let $y_k = \lambda Bx_k$ and $z_k = \lambda Cx_k$, then (\ref{main-algorithm}) is a special case of (\ref{eq3-1}).

Since $B$ is monotone, we have
\begin{equation}\label{eq3-9}
\langle y_k - y_1, x_k -x \rangle = \langle \lambda Bx_k - \lambda Bx, x_k -x \rangle \geq 0.
\end{equation}
Using that $C$ is $\beta$-cocoercive, we obtain
\begin{align}\label{eq3-10}
& \quad 2 \langle z_k - y_2, x - x_{k+1} \rangle \nonumber \\
& = 2 \langle \lambda Cx_k - \lambda Cx, x -x_k   \rangle + 2 \langle \lambda Cx_k - \lambda Cx, x_k - x_{k+1} \rangle \nonumber \\
& \leq - 2\beta \lambda \| Cx_k - Cx \|^2 + \varepsilon_1 \| x_{k+1} - x_k \|^2 + \frac{\lambda^2}{\varepsilon_1}\| Cx_k - Cx \|^2 \nonumber \\
& \quad - \varepsilon_1 \| x_k - x_{k+1} - \frac{\lambda}{\varepsilon_1}(Cx_k - Cx)   \|^2,
\end{align}
for some $\varepsilon_1 >0$. With the help of the Young's inequality ($2ab \leq \varepsilon a^2 + \frac{b^2}{\varepsilon}$, for any $a>0,b>0$ and $\varepsilon >0$), we get
\begin{align}\label{eq3-11}
& \quad 2 \langle y_2 - z_k, y_k - y_{k-1} \rangle \nonumber \\
& =  2 \langle \lambda Cx - \lambda Cx_k, y_k - y_{k-1} \rangle \nonumber \\
& \leq \lambda \varepsilon_2 \| Cx - Cx_k  \|^2 + \frac{\lambda}{\varepsilon_2}\| y_k - y_{k-1} \|^2,
\end{align}
for some $\varepsilon_2 >0$. Further, we obtain
\begin{equation}\label{eq3-12}
2 \langle y_k - y_{k-1}, x_k - x_{k+1} \rangle \leq \varepsilon_3 \|y_k - y_{k-1}\|^2 + \frac{1}{\varepsilon_3}\|x_k - x_{k+1}\|^2, \textrm{ for some } \varepsilon_3 >0,
\end{equation}
and
\begin{equation}\label{eq3-13}
2 \langle y_k - y_{k-1}, x_k - x_{k-1} \rangle \leq \lambda L ( \| x_k - x_{k+1} \|^2 + \| x_k - x_{k-1} \|^2 ).
\end{equation}
Substituting (\ref{eq3-9}), (\ref{eq3-10}), (\ref{eq3-11}), (\ref{eq3-12}) and (\ref{eq3-13}) into (\ref{eq3-2}), we obtain
\begin{align}\label{eq3-14}
\| (x_{k+1} + y_k ) - (x + y_1) \|^2 & \leq \| (x_{k} + y_{k-1} ) - (x + y_1) \|^2 - (1-\varepsilon_1 - \frac{1}{\varepsilon_3}-\lambda L)\| x_k -x_{k+1} \|^2 \nonumber \\
& \quad + \lambda L \| x_k - x_{k-1} \|^2 - \lambda (2\beta - \frac{\lambda}{\varepsilon_1} - \varepsilon_2)\| Cx_k - Cx \|^2 \nonumber \\
& \quad - (3 -\frac{\lambda}{\varepsilon_2} - \varepsilon_3)\| y_k - y_{k-1} \|^2,
\end{align}
which implies that
\begin{align}\label{eq3-15}
\| (x_{k+1} + y_k ) - (x + y_1) \|^2  + (\frac{1}{2}+\alpha)\|x_k - x_{k+1}\|^2 & \leq \| (x_{k} + y_{k-1} ) - (x + y_1) \|^2 + \frac{1}{2}\|x_k -x_{k-1} \|^2 \nonumber \\
& \quad - \lambda (2\beta - \frac{\lambda}{\varepsilon_1} - \varepsilon_2)\| Cx_k - Cx \|^2,
\end{align}
where $\alpha = \frac{1}{2} - (\lambda L + \varepsilon_1 + \frac{1}{\varepsilon_3})>0$.

Let $a_k = x_k + y_{k-1}$ and $a = x + y_1$, (\ref{eq3-15}) is equivalent to
\begin{equation}\label{eq3-16}
\| a_{k+1} - a \|^2  + (\frac{1}{2}+\alpha)\|x_k - x_{k+1}\|^2  \leq \|  a_k - a \|^2 + \frac{1}{2}\|x_k -x_{k-1} \|^2
- \lambda (2\beta - \frac{\lambda}{\varepsilon_1} - \varepsilon_2)\| Cx_k - Cx \|^2.
\end{equation}
The above inequality (\ref{eq3-16}) tells us that
$$
\lim_{k\rightarrow +\infty}( \| a_{k} - a \|^2  + \frac{1}{2}\|x_k - x_{k-1}\|^2 ) \textrm{ exists },
$$
$$
\sum_{k=0}^{+\infty}\| Cx_k - Cx \|^2 < +\infty,
$$
and
$$
\lim_{k\rightarrow +\infty}\| x_k - x_{k+1} \| = 0.
$$
Therefore, the conclusion (i) is true. We also obtain that $\lim_{k\rightarrow +\infty}\| a_{k} - a \|^2$ exists.

(ii) It follows from the Lipschitz continuity of $B$, we have
\begin{align}
\| y_k - y_{k-1} \| & = \| \lambda Bx_k - \lambda Bx_{k-1}\| \nonumber \\
& \leq \lambda L \|x_k - x_{k-1}\| \rightarrow 0 \textrm{ as } k\rightarrow +\infty.
\end{align}
Consequently, we get $\|a_k - a_{k-1}\|\rightarrow 0$ as $k\rightarrow +\infty$. On the other hand, the definition of $a_k$ yields
$$
a_k = (I+\lambda B)x_k + (y_{k-1} - y_k),
$$
and
$$
x_k = J_{\lambda B}(a_k - (y_{k-1}-y_k)).
$$
It follows from $\{z_k\}$ is bounded, $\lim_{k\rightarrow +\infty}\|y_{k-1} - y_k \| =0$ and $J_{\lambda B}$ is nonexpansive that $\{x_k\}$ is also bounded.

Let $(\overline{x},\overline{a})$ be a weak cluster point of the sequence $\{(x_k, a_k)\}$. Then there exists subsequences $\{x_{k_n}\}$ and $\{a_{k_n}\}$ such that $x_{k_n} \rightharpoonup \overline{x}$ and $a_{k_n} \rightharpoonup \overline{a}$. Notice that $Cx_{k_n}\rightarrow Cx$ as $k\rightarrow +\infty$. Since $C$ is cocoercive and its graph is sequentially closed in $H^{weak}\times H^{strong}$, then we have $Cx = C\overline{x}$. Hence $Cx_{k_n}\rightarrow C\overline{x}$ as $k\rightarrow +\infty$. It is observed that the iteration scheme (\ref{main-algorithm}) can be rewritten as
\begin{equation}\label{eq3-18}
-\left(
   \begin{array}{c}
     \lambda Cx_k \\
     0 \\
   \end{array}
 \right) - \left(
             \begin{array}{c}
               a_{k+1}-a_k \\
                a_{k+1}-a_k  \\
             \end{array}
           \right) \in \left(  \left(
                                 \begin{array}{cc}
                                   \lambda A & 0 \\
                                   0 & (\lambda B)^{-1} \\
                                 \end{array}
                               \right) + \left(
                                           \begin{array}{cc}
                                             0 & I \\
                                             -I & 0 \\
                                           \end{array}
                                         \right)
 \right)\left(
          \begin{array}{c}
            a_{k+1}-a_k + x_k \\
            a_{k+1} - x_{k+1} \\
          \end{array}
        \right)
\end{equation}
Since the operator
$$
T : = \left(
                                 \begin{array}{cc}
                                   \lambda A & 0 \\
                                   0 & (\lambda B)^{-1} \\
                                 \end{array}
                               \right) + \left(
                                           \begin{array}{cc}
                                             0 & I \\
                                             -I & 0 \\
                                           \end{array}
                                         \right)
$$
is maximally monotone, and its graph is sequentially closed in $H^{weak}\times H^{strong}$, taking the limit of the subsequence $\{k_n\}$ in (\ref{eq3-18}), we arrive at
\begin{equation*}
\left \{
\begin{aligned}
-\lambda C\overline{x} & \in \lambda A\overline{x} + \overline{a} - \overline{x} \nonumber \\
\overline{x} & \in (\lambda B)^{-1}(\overline{a}-\overline{x}),
\end{aligned}\right.
\end{equation*}
which implies that
\begin{equation*}
\left \{
\begin{aligned}
0 & \in A\overline{x} + B\overline{x} + C\overline{x}  \nonumber \\
\overline{a} & = \overline{x} + \lambda B\overline{x}.
\end{aligned}\right.
\end{equation*}
By Opial's lemma (Lemma \ref{lemma2}), we know that $\{a_k\}$ converges weakly to $\overline{a} = \overline{x} + \lambda B\overline{x}$, where $\overline{x}$ is the weak cluster point of $\{x_k\}$. Notice that $\overline{x}=J_{\lambda B}\overline{a}$, then $J_{\lambda B}\overline{a}$ is the unique cluster point of $\{x_k\}$. Therefore, by Lemma \ref{lemma1}, $\{x_k\}$ converges weakly to a point in $zer(A+B+C)$.

(iii) Finally, it is observed that $y_{k-1} = a_k - x_k \rightharpoonup \overline{a}-\overline{x} = \lambda B\overline{x}$ as $k\rightarrow +\infty$. Together with $\lim_{k\rightarrow +\infty}\| y_k - y_{k-1} \| = 0$, we obtain that $y_k \rightharpoonup \lambda B\overline{x}$, i.e., $Bx_k \rightharpoonup B\overline{x}$.

\end{proof}

\begin{remark}

Since the reflected term $Bx_k - Bx_{k-1}$ is not included in the resolvent of the proposed algorithm (\ref{main-algorithm}), so we name the proposed algorithm as an outer reflected forward-backward splitting algorithm. For the convenience of comparison, we list all these algorithms for solving (\ref{three-sum-monotone}) in Table \ref{comparison-algorithms}.

\begin{table}[htbp]
\footnotesize
\centering
\caption{Iterative algorithms for solving the three-operator monotone inclusion (\ref{three-sum-monotone}).}
\begin{tabular}{c|c|c}
\hline
Methods & $B=0$ &  $C=0$ \\
\hline
\hline
Forward-backward-half forward splitting algorithm (\ref{forward-backward-half-forward-splitting}) \cite{Arias2017A} &  FBS algorithm (\ref{forward-backward-splitting}) &  FBFS algorithm (\ref{forward-backward-forward-splitting})\\
\hline
Semi-forward-reflected-backward splitting algorithm (\ref{Malitsky-Tam-splitting}) \cite{Malitsky2020SIAMJO} &  FBS algorithm (\ref{forward-backward-splitting}) & FRBS  algorithm (\ref{forward-reflected-backward-splitting})\\
\hline
Semi-reflected forward-backward splitting algorithm  (\ref{semi-reflected-forward-backward-splitting}) \cite{Cevher-2019-arxiv} &  FBS algorithm (\ref{forward-backward-splitting}) &  RFBS algorithm (\ref{reflected-forward-backward-splitting})\\
\hline
The proposed algorithm (\ref{main-algorithm}) &  FBS algorithm (\ref{forward-backward-splitting}) &  the iterative algorithm (\ref{Csetnek-splitting})\\
\hline
\end{tabular}\label{comparison-algorithms}
\end{table}

\end{remark}

\begin{remark}
We can reformulate the proposed algorithm (\ref{main-algorithm}) as a fixed point algorithm. In fact, let $\lambda >0$, define
\begin{align*}
M : H\times H \rightarrow H\times H : (x,u) & \mapsto \left(
                                                      \begin{array}{c}
                                                        J_{\lambda A}x \\
                                                        u \\
                                                      \end{array}
                                                    \right) \\
T_1 : H\times H \rightarrow H\times H : (x,u) & \mapsto \left(
                                                          \begin{array}{c}
                                                            x-\lambda Bx - \lambda Cx \\
                                                            \lambda Bx \\
                                                          \end{array}
                                                        \right) \\
T_2 : H\times H \rightarrow H\times H : (x,u) & \mapsto \left(
                                                          \begin{array}{c}
                                                            -\lambda Bx + \lambda u \\
                                                           0 \\
                                                          \end{array}
                                                        \right).
\end{align*}
Then, (\ref{main-algorithm}) can be expressed as the fixed point iteration as follows.
$$
\left(
  \begin{array}{c}
    x_{k+1} \\
    u_{k+1} \\
  \end{array}
\right) = (MT_1 + T_2)\left(
                        \begin{array}{c}
                          x_k \\
                         u_k \\
                        \end{array}
                      \right).
$$
However, we don't know the property of the operator $MT_1 + T_2$. Therefore, we can't obtain the convergence of the proposed algorithm from the perspective of fixed point theory.

\end{remark}

%%%%%%%%%%%%%%%%%%%%%%%%%%%%%%%%%%%%%%%%%%%%%%%%%%%%%%%%%%%%%%%%%%%%%%%%%%%%%%%%%%%%%%%%%%%%%%%%%%55
\section{Applications}

In this section, we apply the proposed algorithm to solve the following primal-dual pair of composite monotone inclusions problem.

 \textbf{Problem 2.} Let $H$ be a real Hilbert space, $z\in H$, $A:H\rightarrow 2^H$ be maximally monotone, $B:H\rightarrow H$ be monotone and $L$-Lipschitzian for some $L>0$, and $C:H\rightarrow H$ be $\beta$-cocoercive for some $\beta>0$. Let $m$ be a strictly positive integer. For any $i=1, \cdots, m$, let $G_i$ be a real Hilbert space, $r_i \in G_i$, $B_i : G_i \rightarrow 2^{G_i}$ be maximally monotone, $D_i: G_i \rightarrow 2^{G_i}$ be maximally monotone and such that $D_{i}^{-1}$ is $\nu_i$-Lipschitzian for some $\nu_i>0$, and $C_i:G_i \rightarrow 2^{G_i}$ be maximally monotone and such that $C_{i}^{-1}$ is $\mu_i$-cocoercive for some $\mu_i>0$. Let $L_i :H\rightarrow G_i $ is a nonzero bounded linear operator. The problem is to solve the primal inclusion:
\begin{equation}\label{composite-primal-inclusion}
\textrm{ find } x\in H, \textrm{ such that } z\in Ax + \sum_{i=1}^{m}L_{i}^{*} ( (B_i \Box D_i \Box C_i)(L_i x - r_i)  ) + Bx + Cx,
\end{equation}
together with the dual inclusion,
\begin{equation}\label{composite-dual-inclusion}
\begin{aligned}
\textrm{ find } v_1\in G_1, \cdots, v_m \in G_m, & \textrm{ such that } (\exists x\in H) \\
z - \sum_{i=1}^{m}L_{i}^{*}v_i & \in Ax + Bx + Cx \\
v_i & \in (B_i \Box D_i \Box C_i)(L_i x - r_i), i = 1,\cdots,m.
\end{aligned}
\end{equation}

Let $C_{i}^{-1} = 0$, for any $i=1, \cdots, m$, and $C=0$, then (\ref{composite-primal-inclusion}) reduces to
\begin{equation}\label{}
\textrm{ find } x\in H, \textrm{ such that } z\in Ax + \sum_{i=1}^{m}L_{i}^{*} ( (B_i \Box D_i )(L_i x - r_i)  ) + Bx,
\end{equation}
which has been studied in \cite{Combettes2012SVVA,Bot2014JMIV}.

Let $D_{i}^{-1} = 0$, for any $i=1, \cdots, m$, and $B=0$, then (\ref{composite-primal-inclusion}) becomes
\begin{equation}\label{}
\textrm{ find } x\in H, \textrm{ such that } z\in Ax + \sum_{i=1}^{m}L_{i}^{*} ( (B_i \Box C_i)(L_i x - r_i)  ) + Cx,
\end{equation}
which has been studied in \cite{vu2013ACM,Combettes2014Optimization}.

In the following theorem, we present the main algorithm for solving Problem 2.

\begin{theorem}\label{main-theorem2}
In the setting of Problem 2. Let $x_0\in H$, $v_{i,0}\in G_i$, $i=1,\cdots, m$, and set
\begin{equation}\label{primal-dual-splitting}
\left \{
\begin{aligned}
x_{k+1} & = J_{\lambda A}( x_k - \lambda (Bx_k +\sum_{i=1}^{m}L_{i}^{*}v_{i,k}) - \lambda Cx_k +\lambda z  ) \\
& \quad - \lambda (  Bx_k + \sum_{i=1}^{m}L_{i}^{*}v_{i,k} - ( Bx_{k-1} + \sum_{i=1}^{m}L_{i}^{*}v_{i,k-1})  ) \\
v_{i,k+1} & = J_{\lambda B_{i}^{-1}}( v_{i,k} - \lambda (-L_i x_k + D_{i}^{-1}v_{i,k}) - \lambda C_{i}^{-1}v_{i,k} - \lambda r_i  ) \\
& \quad - \lambda (-L_i x_k + D_{i}^{-1}v_{i,k} - (-L_i x_{k-1} + D_{i}^{-1}v_{i,k-1})   ), i = 1, \cdots, m,
\end{aligned}\right.
\end{equation}
where $\lambda$ satisfies the condition of Theorem \ref{main-theorem} with $L=\overline{L} = \max\{L, \nu_1, \cdots, \nu_m\} + \sqrt{\sum_{i=1}^{m}\|L_i\|^2}$ and $\beta = \overline{\beta} = \min \{\beta, \mu_1, \cdots, \mu_m\}$. Then there exists $x$ solves (\ref{composite-primal-inclusion}) and $(v_1,  \cdots, v_m)$ solves (\ref{composite-dual-inclusion}) such that $\{x_k\}$ converges weakly to $x$ and $\{(v_{1,k},\cdots,v_{m,k})\}$ converges weakly to $(v_1,  \cdots, v_m)$.
\end{theorem}

\begin{proof}
We first reformulate the primal inclusion (\ref{composite-primal-inclusion}) and the dual inclusion (\ref{composite-dual-inclusion}) as the formulation of (\ref{three-sum-monotone}) in a suitable Hilbert space. Define the Hilbert space $\mathcal{K} = H\times G_1 \times \cdots G_m$, which equipped with the scalar product
$$
\langle (x,v_1, \cdots, v_m), (y,w_1,\cdots,w_m)   \rangle_{\mathcal{K}} = \langle x,y\rangle + \sum_{i=1}^{m}\langle v_, w_i\rangle,
$$
and the associated norm on $K$ is
$$
\| (x,v_1, \cdots, v_m) \|_\mathcal{K} = \sqrt{\|x\|^2 + \sum_{i=1}^{m}\|v_i\|^2 },
$$
for any $(x,v_1, \cdots, v_m)\in \mathcal{K}$ and $(y,w_1,\cdots,w_m)\in \mathcal{K}$.

Let the operators
\begin{align*}
M:\mathcal{K}\rightarrow 2^{\mathcal{K}}: (x,v_1, \cdots, v_m) & \mapsto (-z+Ax) \times (r_1 + B_{1}^{-1}v_1) \times \cdots \times (r_m + B_{m}^{-1}v_m) \\
Q:\mathcal{K}\rightarrow \mathcal{K}: (x,v_1, \cdots, v_m) & \mapsto (Bx+\sum_{i=1}^{m}L_{i}^{*}v_i, -L_1 x +D_{1}^{-1}v_1, \cdots, -L_m x +D_{m}^{-1}v_m) \\
R:\mathcal{K}\rightarrow \mathcal{K}: (x,v_1, \cdots, v_m) & \mapsto (Cx,C_{1}^{-1}v_1, \cdots, C_{m}^{-1}v_m).
\end{align*}

It follows from Theorem 3.1 of \cite{Combettes2012SVVA}, we know that $M$ is maximally monotone, and $J_{\lambda M}(x,v_1, \cdots, v_m) = (J_{\lambda A}(x+\lambda z), J_{\lambda B_{1}^{-1}}(v_1 - \lambda r_1),\cdots, J_{\lambda B_{m}^{-1}}(v_m - \lambda r_m) )$, $\forall \lambda >0$. Further, $Q$ is monotone and $\overline{L}$-Lipschitzian, where $\overline{L} = \max\{L, \nu_1, \cdots, \nu_m\} + \sqrt{\sum_{i=1}^{m}\|L_i\|^2}$. Next, we prove that $R$ is $\overline{\beta}$-cocoercive, where $\overline{\beta} = \min \{\beta, \mu_1, \cdots, \mu_m\}$. In fact, let $(x,v_1, \cdots, v_m) \in \mathcal{K}$ and $(y,w_1, \cdots, w_m) \in \mathcal{K}$, it follows from the cocoercivity of $C$, we have
\begin{align*}
& \quad \langle (x,v_1, \cdots, v_m)- (y,w_1, \cdots, w_m), R(x,v_1, \cdots, v_m) - R(y,w_1, \cdots, w_m) \rangle_{\mathcal{K}} \\
& = \langle x-y, Cx-Cy \rangle + \sum_{i=1}^{m}\langle v_i - w_i, C_{i}^{-1}v_i -C_{i}^{-1}w_i \rangle \\
& \geq \beta \|Cx - Cy \|^2 + \sum_{i=1}^{m}\mu_i \| C_{i}^{-1}v_i - C_{i}^{-1}w_i \|^2 \\
& \geq \overline{\beta} \| R(x,v_1, \cdots, v_m) - R(y,w_1, \cdots, w_m) \|^2,
\end{align*}
where $\overline{\beta} = \min \{\beta, \mu_1, \cdots, \mu_m\}$.

Let $x$ solves (\ref{composite-primal-inclusion}), we have
\begin{align*}
& \quad z\in Ax + \sum_{i=1}^{m}L_{i}^{*} ( (B_i \Box D_i \Box C_i)(L_i x - r_i)  ) + Bx + Cx \\
& \Leftrightarrow \exists (x,v_1, \cdots, v_m)\in \mathcal{K}\quad    z \in  Ax + \sum_{i=1}^{m}L_{i}^{*}v_i + Bx + Cx \\
& \qquad \qquad \qquad \qquad \qquad  v_i \in  (B_i \Box D_i \Box C_i)(L_i x - r_i), i= 1,\cdots, m \\
& \Leftrightarrow \exists (x,v_1, \cdots, v_m)\in \mathcal{K}\quad    0 \in -z +  Ax + \sum_{i=1}^{m}L_{i}^{*}v_i + Bx + Cx \\
& \qquad \qquad \qquad \qquad \qquad  0 \in  r_i +  B_{i}^{-1}v_i + D_{i}^{-1}v_i +  C_{i}^{-1}v_i - L_i v_i), i= 1,\cdots, m \\
& \Leftrightarrow \exists (x,v_1, \cdots, v_m)\in \mathcal{K}\quad (0,\cdots, 0) \in (-z+Ax) \times (r_1 + B_{1}^{-1}v_1) \times \cdots \times (r_m + B_{m}^{-1}v_m) \\
&  \qquad \qquad \qquad \qquad + (Bx+\sum_{i=1}^{m}L_{i}^{*}v_i, -L_1 x +D_{1}^{-1}v_1, \cdots, -L_m x +D_{m}^{-1}v_m) \\
&  \qquad \qquad \qquad \qquad + (Cx,C_{1}^{-1}v_1, \cdots, C_{m}^{-1}v_m).
\end{align*}

Therefore, we obtain
\begin{align*}
& x \textrm{ solves } (\ref{composite-primal-inclusion}) \textrm{ and } (v_1,\cdots, v_m)\in G_1\times \cdots, G_m \textrm{ solves } (\ref{composite-dual-inclusion}) \\
& \Leftrightarrow (0,\cdots, 0) \in (M+Q+R)(x,v_1, \cdots, v_m).
\end{align*}

 Then, we can apply the proposed iterative algorithm (\ref{main-algorithm}), and the resulting algorithm is presented below,
\begin{equation*}
\begin{aligned}
(x_{k+1}, v_{1,k+1}, \cdots, v_{m,k+1}) & = J_{\lambda M}((x_{k}, v_{1,k}, \cdots, v_{m,k}) - \lambda Q (x_{k}, v_{1,k}, \cdots, v_{m,k}) - \lambda R (x_{k}, v_{1,k}, \cdots, v_{m,k})   ) \\
& \quad - \lambda ( Q(x_{k}, v_{1,k}, \cdots, v_{m,k}) - Q(x_{k-1}, v_{1,k-1}, \cdots, v_{m,k-1}) ),
\end{aligned}
\end{equation*}
which is equivalent to (\ref{primal-dual-splitting}). By Theorem \ref{main-theorem}, we can obtain the conclusions of Theorem \ref{main-theorem2}.

\end{proof}

\begin{remark}
According to the proof of Theorem \ref{main-theorem2}, we can also use the FBHFS algorithm (\ref{forward-backward-half-forward-splitting}), the SFRBS algorithm (\ref{Malitsky-Tam-splitting}), and the SRFBS algorithm (\ref{semi-reflected-forward-backward-splitting}) to solve Problem 2. We will separate study the performance of these algorithms for solving different convex optimization problems arising in signal and image processing.

\end{remark}

%%%%%%%%%%%%%%%%%%%%%%%%%%%%%%%%%%%%%%%%%%%%%%%%%%%%%%%%%%%%%%%%%%%%%%%%%%%%%%%%%%%%%%%%%%%%%%%%%%%%%%%%%%%%%%%%%%%%%%%%%%%%%%%%%%%%%%%%%%
\section{Conclusions and future works}

In this paper, we proposed an outer reflected forward-backward splitting algorithm for solving the monotone inclusions problem (\ref{three-sum-monotone}). The convergence of the proposed algorithm is studied in a general Hilbert space. By using a standard product space approach, we reformulated the primal-dual pair of (\ref{composite-primal-inclusion}) and (\ref{composite-dual-inclusion}) as the problem of finding the zeros of the sum of a maximally monotone operator, a monotone Lipschitzian operator, and a cocoercive operator. A primal-dual splitting algorithm is obtained by making use of the proposed algorithm. Finally, we list some possible generalizations of the proposed algorithm.

(i) Inertial and relaxation terms. The inertial and relaxation terms are useful to accelerate operator splitting algorithm. See, for example \cite{Attouch2019MP,Attouch2019AMO}. It is desired to combine the inertial and relaxation methods with the proposed algorithm.

(ii) Inexact case. In practice, the resolvent operator $J_{\lambda A}$, the operators $B$ and $C$ may not exactly be computed. Whether we can study the convergence of the proposed algorithm with errors or not?

(iii) Four-operator monotone inclusions and beyond. It is natural to extend the considered three-operator monotone inclusions to the following four-operator monotone inclusions,
\begin{equation}\label{four-operator}
\textrm{ find } x\in H, \textrm{ such that } 0\in A_1x + A_2x + Bx + Cx,
\end{equation}
where $A_1:H\rightarrow 2^H$ and $A_2:H\rightarrow 2^H$ are maximally monotone operators, $B$ and $C$ are the same as (\ref{three-sum-monotone}). In particular, when $C=0$, (\ref{four-operator}) reduces to
\begin{equation}\label{}
\textrm{ find } x\in H, \textrm{ such that } 0\in A_1x + A_2x + Bx,
\end{equation}
which has been studied in \cite{Briceno2015JOTA,Rieger2020AMC,Ryu2020JOTA}. If $B=0$, (\ref{four-operator}) becomes
\begin{equation}\label{}
\textrm{ find } x\in H, \textrm{ such that } 0\in A_1x + A_2x + Cx,
\end{equation}
which has been studied in \cite{Raguet-SIAM-2013,briceno2015Optim,davis2015,Raguet2019OL}.

It would be interesting to develop an effective algorithm with a complete splitting structure for solving (\ref{four-operator}), which may combine the three-operator splitting algorithms mentioned above. Furthermore, can we develop an algorithm for solving the following monotone inclusions?
\begin{equation}\label{}
\textrm{ find } x\in H, \textrm{ such that } 0\in \sum_{i=1}^{m}A_ix  + Bx + Cx,
\end{equation}
where $A_i:H\rightarrow 2^H, i= 1,\cdots, m$ are maximally monotone operators, $B$ and $C$ are the same as (\ref{three-sum-monotone}).

%%%%%%%%%%%%%%%%%%%%%%%%%%%%%%%%%%%%%%%%%%%%%%%%%%%%%%%%%%%%%%%%%%%%%%%%%%%%%%%%%%%%%%%%%%%%%%%%%%%%%%%%%%%%%%%%%%%%%%%%%%%%%%%%%%%%%%%%%%%%%%%%%%%%%
\section*{Acknowledgement}

This research was funded by the National Natural Science Foundation of China (12061045, 11661056, 11401293), the China Postdoctoral Science Foundation (2015M571989) and the Jiangxi Province Postdoctoral Science Foundation (2015KY51).

% References
%\bibliographystyle{unsrt}
%\bibliography{klreference-en,essayfirst-en} %

\begin{thebibliography}{10}

\bibitem{Lions1979SIAM}
P.L. Lions and B.~Mercier.
\newblock Splitting algorithms for the sum of two nonlinear operators.
\newblock {\em SIAM J. Numer. Anal.}, 16:964--979, 1979.

\bibitem{passty1979JMAA}
G.B. Passty.
\newblock Ergodic convergence to a zero of the sum of monotone operators in
  hilbert space.
\newblock {\em J. Math. Anal. Appl.}, 72(2):383--390, 1979.

\bibitem{Chenhg1997}
George~H.G. Chen and R.T. Rockafellar.
\newblock Convergence rates in forward-backward splitting.
\newblock {\em SIAM J. Optim.}, 7(2):421--444, 1997.

\bibitem{combettes2005}
P.L. Combettes and V.R. Wajs.
\newblock Signal recovery by proximal forward-backward splitting.
\newblock {\em Multiscale Model. Simul.}, 4:1168--1200, 2005.

\bibitem{Combettes2014Optimization}
P.L. Combettes and B.C. V{{\~u}}.
\newblock Variable metric forward-backward splitting with applications to
  monotone inclusions in duality.
\newblock {\em Optim.}, 63(9):1289--1318, 2014.

\bibitem{Combettes2015JMAA}
P.L. Combettes and I.~Yamada.
\newblock Compositions and convex combinations of averaged nonexpansive
  operators.
\newblock {\em J. Math. Anal. Appl.}, 425:55--70, 2015.

\bibitem{lorenz2015JMIV}
D.A. Lorenz and T.~Pock.
\newblock An inertial forward-backward algorithm for monotone inclusions.
\newblock {\em J. Math. Imaging Vis.}, 51:311--325, 2015.

\bibitem{CuiJIA2019}
F.Y. Cui, Y.C. Tang, and C.X. Zhu.
\newblock Convergence analysis of a variable metric forward-backward splitting
  algorithm with applications.
\newblock {\em J. Inequal. Appl.}, 2019:141, 2019.

\bibitem{Raguet-SIAM-2013}
H.~Raguet, J.~Fadili, and G.~Peyr\'{e}.
\newblock A generalized forward-backward splitting.
\newblock {\em SIAM J. Imaging Sci.}, 6(3):1199--1226, 2013.

\bibitem{Raduet2015SIAMJIS}
H.~Raguet and L.~Landrieu.
\newblock Preconditioning of a generalized forward-backward splitting and
  application to optimization on graphs.
\newblock {\em SIAM J. Imaging Sci.}, 8(4):2706--2739, 2015.

\bibitem{briceno2015Optim}
L.M. Brice{\~{n}}o-Arias.
\newblock Forward-douglas-rachford splitting and forward-partial inverse method
  for solving monotone inclusions.
\newblock {\em Optimization}, 64:1239--1261, 2015.

\bibitem{davis2015}
D.~Davis and W.T. Yin.
\newblock A three-operator splitting scheme and its optimization applications.
\newblock {\em Set-Valued Var. Anal.}, 25(4):829--858, 2017.

\bibitem{Zong2018}
C.X. Zong, Y.C. Tang, and Y.J. Cho.
\newblock Convergence analysis of an inexact three-operator splitting
  algorithm.
\newblock {\em Symmetry}, 10(11):563, 2018.

\bibitem{Tseng2000SIAM}
P.~Tseng.
\newblock A modified forward-backward splitting method for maximal monotone
  mappings.
\newblock {\em SIAM J. Control Optim.}, 38(2):431--446, 2000.

\bibitem{Dong2003}
Y.D. Dong.
\newblock {\em Splitting methods for monotone inclusions}.
\newblock PhD thesis, 2003.

\bibitem{Dong2014AMC}
Y.Y. Huang and Y.D.Dong.
\newblock New properties of forward-backward splitting and a practical
  proximal-descent algorithm.
\newblock {\em Appl. Math. Comput.}, 237:60--68, 2014.

\bibitem{briceno2011SIAM}
L.M. Brice{\~{n}}o-Arias and P.L. Combettes.
\newblock A monotone+skew splitting splitting model for composite monotone
  inclusions in duality.
\newblock {\em SIAM J. Control Optim.}, 21(4):1230--1250, 2011.

\bibitem{Bot2016NA}
R.I. Bo\c{t} and E.R. Csetnek.
\newblock An inertial forward-backward-forward primal-dual splitting algorithm
  for solving monotone inclusion problems.
\newblock {\em Numer. Algorithms}, 71:519--540, 2016.

\bibitem{Gibali2020CAM}
A.~Gibali, D.V. Thong, and N.T. Vinh.
\newblock Three new iterative methods for solving inclusion problems and
  related problems.
\newblock {\em Computational and Applied Mathematics}, 39: 187,
  2020.

\bibitem{Vu2013NFAO}
B.C. V{\~{u}}.
\newblock A variable metric extension of the forward-backward-forward algorithm
  for monotone operators.
\newblock {\em Numer. Funct. Anal. Optim.}, 34(9):1050--1065, 2013.

\bibitem{Vu2016OL}
B.C. V{\~{u}}.
\newblock Almost sure convergence of the forward-backward-forward splitting
  algorithm.
\newblock {\em Optim. Lett.}, 10:781--803, 2016.

\bibitem{Arias2017A}
L.M. Brice{\~{n}}o-Arias and D.~Davis.
\newblock Forward-backward-half forward algorithm for solving monotone
  inclusions.
\newblock {\em SIAM J. Optim.}, 28:2839--2871, 2018.

\bibitem{Malitsky2020SIAMJO}
Y.~Malitsky and M.K. Tam.
\newblock A forward-backward splitting method for monotone inclusions without
  cocoercivity.
\newblock {\em SIAM J. Optim.}, 30(2):1451--1472, 2020.

\bibitem{Cevher-2019-arxiv}
V.~Cevher and B.C. V{\~{u}}.
\newblock A refected foward-backward splitting method for monotone inlcusions
  involving lipschitzian operators.
\newblock {\em arXiv eprint}, arXiv:1908.05912, 2019.

\bibitem{Cevher2020SVVA}
V.~Cevher and B.C. V{\~{u}}.
\newblock A reflected forward-backward splitting method for monotone inclusions
  involving lipschitzian operators.
\newblock {\em Set-Valued Var. Anal.}, 2020.

\bibitem{Gibali2020Symmetry}
A.~Gibali and Y.~Shehu.
\newblock A symmetric fbf method for solving monotone inclusions.
\newblock {\em Symmetry}, 12(9):1456, 2020.

\bibitem{Hieu2020BIMS}
D.V. Hieu, L.V. Vy, and P.K. Quy.
\newblock Three-operator splitting algorithm for a class of variational
  inclusion problems.
\newblock {\em Bulletin of the Iranian Mathematical Society}, 46:1055--1071,
  2020.

\bibitem{Hieu20204OR}
D.V. Hieu, P.K. Anh, and L.D. Muu.
\newblock Modified forward-backward splitting method for variational
  inclusions.
\newblock {\em 4OR-Q J Oper Res}, 2020.

\bibitem{Csetnek2019AMO}
E.R. Csetnek, Y.~Malitsky, and M.K. Tam.
\newblock Shadow douglas-rachford splitting for monotone inclusions.
\newblock {\em Appl. Math. Optim.}, 80:665--678, 2019.

\bibitem{bauschkebook2017}
H.H. Bauschke and P.L. Combettes.
\newblock {\em Convex Analysis and Monotone Operator Theory in Hilbert Spaces}.
\newblock Springer, London, second edition, 2017.

\bibitem{Combettes2012SVVA}
P.L. Combettes and J.-C. Pesquet.
\newblock Primal-dual splitting algorithm for solving inclusions with mixtures
  of composite, lipschitzian, and paralle-sum type monotone operators.
\newblock {\em Set-Valued Var. Anal.}, 20(2):307--330, 2012.

\bibitem{Bot2014JMIV}
R.I. Bo\c{t} and C.~Hendrich.
\newblock Convergence analysis for a primal-dual monotone+skew splitting
  algorithm with applications to total variation minimization.
\newblock {\em J. Math. Imaging Vis.}, 49(3):551--568, 2014.

\bibitem{vu2013ACM}
B.C. V{\~{u}}.
\newblock A splitting algorithm for dual monotone inclusions involving
  cocoercive operators.
\newblock {\em Adv. Comput. Math.}, 38:667--681, 2013.

\bibitem{Attouch2019MP}
H.~Attouch and A.~Cabot.
\newblock Convergence of a relaxed inertial proximal algorithm for maximally
  monotone operators.
\newblock {\em Math. Program.}, 2019.

\bibitem{Attouch2019AMO}
H.~Attouch and A.~Cabot.
\newblock Convergence of a relaxed inertial forward-backward algorithm for
  structured monotone inclusions.
\newblock {\em Appl. Math. Optim.}, 80:547--598, 2019.

\bibitem{Briceno2015JOTA}
L.M.~Brice\ {n}o Arias.
\newblock Foward-partial inverse forward splitting for solving monotone
  inclusions.
\newblock {\em J. Optim. Theory Appl.}, 166:391--413, 2015.

\bibitem{Rieger2020AMC}
J.~Rieger and M.K. Tam.
\newblock Backward-forward-reflected-backward splitting for three operator
  monotone inclusions.
\newblock {\em Appl. Math. Comput.}, 381:125248, 2020.

\bibitem{Ryu2020JOTA}
E.K. Ryu and B.C. V{\~{u}}.
\newblock Finding the forward-douglas-rachford-forward method.
\newblock {\em J. Optim. Theory Appl.}, 184:858--876, 2020.

\bibitem{Raguet2019OL}
H.~Raguet.
\newblock A note on the forward-douglas¨crachford splitting for monotone
  inclusion and convex optimization.
\newblock {\em Optim. Lett.}, 13(4):717--740, 2019.

\end{thebibliography}

\end{document}